\documentclass[12pt,twoside,reqno]{amsart}
\usepackage[colorlinks=true,citecolor=blue]{hyperref}
\usepackage{mathptmx, amsmath, amssymb, amsfonts, amsthm, mathptmx, enumerate, color}
\usepackage{mathptm}
\usepackage[table]{xcolor}
\usepackage{graphicx}
\usepackage{epstopdf}

\newtheorem{theorem}{Theorem}[section]

\newtheorem{lemma}{Lemma}[section]

\theoremstyle{definition}
\newtheorem{definition}{Definition}[section]

\numberwithin{equation}{section}

\begin{document}
\setcounter{page}{1}

\vspace*{1.0cm}
\title[MODIFIED S-ITERATION PROCESS...]
{INERTIAL MODIFIED S-ITERATION PROCESS FOR SPLIT MONOTONE INCLUSION AND FIXED POINT PROBLEM IN REAL HILBERT SPACE}
\author[S. Husain, U. Rafat]{ Shamshad Husain$^{1}$, Uqba Rafat$^{2,*}$}
\maketitle
\vspace*{-0.6cm}

\begin{center}
{\footnotesize {\it

$^1$Department of Applied Mathematics, Zakir Husain College of Engineering and Technology, Aligarh Muslim University, Aligarh 202002, India\\
$^2$Department of Applied Mathematics, Zakir Husain College of Engineering and Technology, Aligarh Muslim University, Aligarh 202002, India

}}\end{center}

\vskip 4mm {\small\noindent {\bf Abstract.}
In this article, we present a modified S-iteration process that we combine with inertial extrapolation to find a common solution to the split monotone inclusion problem and the fixed point problem in real Hilbert space. Our goal is to establish a strong convergence theorem for approximating a common solution. Under some mild conditions, the problem can be solved. We also provide a numerical example to show that our algorithm's acceleration works well.

 }

\renewcommand{\thefootnote}{}
\footnotetext{ $^*$Corresponding author.
\par
E-mail addresses: s\_husain68@gmail.com (Shamshad Husain), uqba.1709@gmail.com (Uqba Rafat).
\par
Received... ; Accepted ... }

\section{Introduction}

Let $\mathcal{H}$ be a real Hilbert space with scalar product $\langle\cdot,\cdot \rangle$ and induced norm $\|\cdot \|$. A mapping $ \mathcal{S}: \mathcal{H}\rightarrow\mathcal{H}$ is said to be a nonexpansive if
\begin{align*}
\|\mathcal{S}x-\mathcal{S}y\|\leq\|x-y\| ~~~~~\forall x,y\in \mathcal{H}
\end{align*}
and we denote the set of fixed point of $\mathcal{S}$ by $\mathcal{F}(\mathcal{S}):=\{x \in \mathcal{H}~|~ \mathcal{S}x=x\}$. We say a mapping $ \mathcal{S}: \mathcal{H}\rightarrow\mathcal{H}$ is\\
$(i)$ monotone if
\begin{align*}
\langle\mathcal{S}x-\mathcal{S}y,x-y\rangle\geq 0 ~~~~~\forall x,y\in \mathcal{H};
\end{align*}
$(ii)$ $\alpha$-strongly monotone if there exist a constant $\alpha>0$ such that
\begin{align*}
\langle\mathcal{S}x-\mathcal{S}y,x-y\rangle\geq \alpha\|x-y\|^2 ~~~~~\forall x,y\in \mathcal{H};
\end{align*}
$(iii)$ $\beta$- inverse strongly monotone $(\beta-ism)$ if there exist a constant $\beta>0$ such that
\begin{align*}
\langle\mathcal{S}x-\mathcal{S}y,x-y\rangle\geq \beta \|\mathcal{S}x-\mathcal{S}y\|^2 ~~~~~\forall x,y\in \mathcal{H};
\end{align*}
$(iv)$ firmly nonexpansive if 
\begin{align*}
\langle\mathcal{S}x-\mathcal{S}y,x-y\rangle\geq \|\mathcal{S}x-\mathcal{S}y\|^2 ~~~~~\forall x,y\in \mathcal{H};
\end{align*}
\indent A set-valued mapping $\mathcal{M}:\mathcal{H}\rightarrow 2^\mathcal{H}$ is called monotone if for all $ x,y \in \mathcal{H}$ with $u\in \mathcal{M}(x)$ and $v \in \mathcal{M}(y)$ then 
\begin{align*}
\langle x-y,u-v\rangle\geq 0. 
\end{align*}
A monotone mapping $\mathcal{M}$ is said to be maximal if the graph of $\mathcal{M}$, denoted as   $G(\mathcal{M})=\{(x,y) : y \in \mathcal{M}(x)\}$ is not properly contained in the graph of any other monotone multi-valued mapping and the resolvent operator $\mathfrak{J}^\mathcal{M}_\lambda$ associated with $\mathcal{M}$ and $\lambda$ is the mapping $\mathfrak{J}^\mathcal{M}_\lambda : \mathcal{H} \rightarrow \mathcal{H}$ defined by
\begin{align}
\mathfrak{J}^\mathcal{M}_\lambda(x) := (I+\lambda\mathcal{M})^{-1}(x), ~~ x\in \mathcal{H},~~ \lambda>0.
\end{align}
It is known that the resolvent operator $\mathfrak{J}^\mathcal{M}_\lambda$ is\\
$(i)$ Single-valued\\
$(ii)$ Nonexpansive\\
$(iii)$ 1-inverse monotone(see, for eg.,\cite{1}).\\
\rm \indent Let $\mathcal{H}_1$ and $\mathcal{H}_2$ be real Hilbert spaces and $g_1 : \mathcal{H}_1 \rightarrow \mathcal{H}_2$, $g_2 : \mathcal{H}_1 \rightarrow \mathcal{H}_2$ be $\mu$-inverse strongly monotone and $\nu$-inverse strongly monotone mappings, respectively. Let $\mathcal{B}_1 :\mathcal{H}_1\rightarrow 2^{\mathcal{H}_1}$, $\mathcal{B}_2:\mathcal{H}_2\rightarrow 2^{\mathcal{H}_2}$ be the maximal monotone mappings and $\mathcal{A}:\mathcal{H}_1\rightarrow \mathcal{H}_2$ be a bounded linear operator. Then the Split monotone inclusion problem \textbf{(SMIP)} is to find $x^* \in \mathcal{H}_1$ such that
\begin{align}
0 \in g_1(x^*)+\mathcal{B}_1(x^*)\label{1.2}
\end{align}
and $y^*= \mathcal{A}x^* \in \mathcal{H}_2$ solves
\begin{align}
0\in g_2(y^*)+\mathcal{B}_2(y^*)\label{1.3}
\end{align}
We shall denoted the soluton set of \textbf{(SMIP)} (\ref{1.2})-(\ref{1.3}) by $\Omega$, that is
\begin{equation*}
\Omega := \{x^*\in \mathcal{H}_1:0\in g_1(x^*)+\mathcal{B}_1(x^*)~~\rm and~~y^*= \mathcal{A}x^* \in \mathcal{H}_2 \rm~~\rm such~~that~~0\in g_2(y^*)+\mathcal{B}_2(y^*)\}.
\end{equation*}
\indent In \cite{16}, Moudafi introduced \textbf{(SMIP)}(\ref{1.2})-(\ref{1.3}) and proposed an iterative method for
solving \textbf{(SMIP)}(\ref{1.2})-(\ref{1.3}) and generalizes the split fixed
point problem, split variational inequality problem, split zero problem and split
feasibility problem, which have been studied extensively by many authors (see,
e.g.\cite{5}-\cite{10} and applied to solving many real life problems such as in modelling intensity-modulated radiation therapy treatment planning (see \cite{8,9}), modelling
of inverse problems arising from phase retrieval and in sensor networks in
computerized tomography and data compression (see \cite{11,12}).\\
\indent Byrne et al.\cite{31} using the following iterative scheme: for a given $x_0 \in \mathcal{H}_1$ the sequence $\{x_n\}$ generated iteratively by:
\begin{align*}
x_{n+1}=\mathfrak{J}^{'}(x_n),
\end{align*}
where, $\mathfrak{J}^{'}=\mathfrak{J}^{\mathcal{B}_1}_\lambda(I+\gamma\mathcal{A}^*(\mathfrak{J}^{\mathcal{B}_2}_\lambda-I)\mathcal{A}), ~~\lambda>0,~~\gamma\in \Big(0,\frac{2}{\|A^*A\|}\Big)$, 
and obtain convergence results on \textbf{(SMIP)} (\ref{1.2})-(\ref{1.3}).\\
\indent Suparatulatron et al.\cite{24} introduced modified S-iteration process defined as follows:
$x_0 \in \mathcal{C}$ 
\begin{align}
y_n=(1-\beta_n)x_n+\beta_n\mathcal{S}_1(x_n),\nonumber\\
x_{n+1}=(1-\alpha_n)\mathcal{S}_1(x_n)+\beta_n\mathcal{S}_2(x_n)
\end{align} 
$n\geq0$, where $\mathcal{C}$ is a non-empty subset of real Hilbert space $\mathcal{H}$ and the sequences $\{\alpha_n\}$ and $\{\beta_n\}$ are in the interval $(0,1)$ and $\mathcal{S}_1,\mathcal{S}_2 :\mathcal{C}\rightarrow \mathcal{C}$ are the nonexpansive mappings. Using some given conditions, they proved weak and strong convergence theorems of this iteration
process for finding common fixed points of two nonexpansive mappings. They also provided an example from a numerical experiment which
supported the idea that the sequence generated by the modified S-iteration converges
faster than the one generated by an Ishikawa iteration. So, to obtain a faster algorithm
revised from a modified S-iteration process, it should be combined with the inertial exatrapolation as well.\\
\indent The idea of modified S-iteration process motivated us and we combined inertial extrapolation with modified S-iteration to capture a common solution of \textbf{(SMIP)}(\ref{1.2})-(\ref{1.3}) and $\mathcal{F}(\mathcal{S})$ for the nonexpansive mappings $\mathcal{S}$ and $\mathfrak{J}^{'}$ in the real Hilbert space. Under some assumptions, we also test the weak and strong convergence of our iteration to find the common solution to our problem.

\section{Preliminaries}

This section deals with some useful definition and lemmas which will be used frequently in the next section. The following identity appears several times through out the paper:
\begin{align}
\|\alpha x+(1-\alpha)y\|^2=\alpha\|x\|^2+(1-\alpha)\|y\|^2-\alpha(1-\alpha)\|x-y\|^2,\label{2.1}
\end{align}
for all $\alpha \in \mathbb{R}$ and x,y $\in \mathcal{H}$ 
\begin{definition}
	\rm A Hilbert space $\mathcal{H}$ is said to have Opial's property if whenever a sequence $\{x_n\}$ in $\mathcal{H}$ converges weakly to x, then 
\end{definition}
\begin{align*}
\liminf_{n\rightarrow \infty}\|x_n-x\|\leq \liminf_{n\rightarrow \infty} \|x_n-y\|, 
\end{align*}
for all $y \in \mathcal{H} , y \neq x$.
\begin{definition}\rm\cite{25}
	\rm Let C be a subset of a metric space $(X,d)$. A mapping $T :C \rightarrow C$ is  $\textbf{semicompact}$ if for a sequence $\{x_n\}$ in C with $ \lim_{n\rightarrow \infty} d(x_n,T(x_n))=0$, there exist a subsequence $\{x_{n_k}\} $ of $\{x_n\}$ such that $x_{n_k}\rightarrow p\in C$
\end{definition}
\begin{definition}
	\rm Let $ \{u_n\}$ and $\{v_n\}$ be two sequences of positive number that converges to u and v, respectively. Assume there exists
	\begin{align*}
	\lim_{n\rightarrow \infty}\frac{|u_n-u|}{|v_n-v|}=l
	\end{align*}
	(1) If $l=0$, then it is said that the sequence $\{u_n\}$ converges to u faster than the sequence $\{v_n\}$ converges to v.\\
	(2) if $0<l<\infty$, then we say that the sequence $\{u_n\}$ and $\{v_n\}$ having the same rate of convergence.
\end{definition}
\begin{definition}
	\rm\cite{25} Let C be a nonempty closed convex subset of real Hilbert space $\mathcal{H}$.The mapping $\mathcal{S}$ and $\mathfrak{J}^{'}$ on C are said to satisfy \textbf{condition (B)} if there exists a nondecreasing function $f :[0,\infty)\rightarrow [0,\infty)$ with $f(0)=0$ and $f(r)>0$ for $r>0$ such that, for all $x\in C$, 
\end{definition}
$~~~~~~~~~~~~~~~~~~~~~~~~~~~~~~~~~\rm max\{\|x-\mathcal{S}(x)\|, \|x-\mathfrak{J}^{'}(x)\|\}\geq f(d(x, \Sigma)),$\\
where, we denote $\Sigma=\mathcal{F}(\mathcal{S})\bigcap\Omega$.
\begin{lemma}\label{n}
\cite{26} Let $\mathcal{H}$ be a real Hilbert space, and $\alpha_n$ be a sequence in $[\delta,1-\delta]$ ~for $\delta \in(0,1)$. Suppose that sequence $\{x_n\}$ and $\{y_n\}$ in $\mathcal{H}$ satisfies the following:
\end{lemma}
\begin{enumerate}
\item $\liminf_{n\rightarrow \infty} \|x_n\|\leq c$,
\item $\liminf_{n\rightarrow \infty} \|y_n\|\leq c$ and
\item $\liminf_{n\rightarrow \infty} \|\alpha_nx_n+(1-\alpha_n)y_n\|=c$ for some $c\geq0$,\\
\end{enumerate}
Then, $\liminf_{n\rightarrow \infty} \|x_n-y_n\|=0.$
\begin{lemma}\label{2.2}
\cite{28} Let $\mathcal{H}$ be a real Hilbert space that has \textbf{Opial's property} and let $\{x_n\}$ be a sequence in $\mathcal{H}$. Let x,y in $\mathcal{H}$ be such that $\lim_{n\rightarrow \infty}\|x_n-x\|$ and $\lim_{n\rightarrow \infty}\|x_n-y\|$ exist. If $\{x_{n_j}\}$ and $\{x_{n_k}\}$ are the subsequences of $\{x_n\}$ that converges to x and y, respectively, then x=y.
\end{lemma}

\begin{lemma}\cite{29} \label{l,2.3}Let $\{\psi_n\}, \{\delta_n\}$ and $\{\alpha_n\}$ be sequences in $[0,\infty)$ such that $\psi_{n+1}\leq\psi_n+\alpha_n(\psi_n-\psi_{n-1})+\delta_n$ for all $n\geq1$, $\sum_{n=1}^{\infty}\delta_n<\infty$ and there exist a real number $\alpha$ with $0\leq\alpha_n\leq\alpha<1$ for all $n\geq1$. Then the following hold: 
\end{lemma}

\begin{enumerate}
	\item $\sum_{n\geq1}[\psi_n-\psi_{n-1}]_+<\infty$ where $[t]_+=\rm max\{t,0\}.$
	\item there exists $\psi^* \in [0,\infty)$ such that $\lim_{n\rightarrow \infty}\psi_n=\psi^*.$\label{2.3}
\end{enumerate}

\begin{lemma}\cite{30} 
	Let C be a nonempty set of a real Hilbert space $\mathcal{H}$ and $\{x_n\}$ be a sequence in $\mathcal{H}$ such that the following two conditions hold:
\end{lemma}

\begin{enumerate}
	\item for any $x \in C$, $\lim_{n\rightarrow \infty}\|x_n-x\|$ exists;
	\item every sequential weak cluster point of $\{x_n\}$ is in C. 
\end{enumerate}

\begin{lemma}\label{2.5}\cite{30} Let C be a nonempty closed convex subset of a real Hilbert space $\mathcal{H}$, $\mathcal{S}:C \rightarrow \mathcal{H}$ a nonexpansive mapping. Let $\{x_n\}$ be a sequence in C and $x\in\mathcal{H}$ such that $x_n\rightharpoonup x$ and $\mathcal{S}x_n-x_n\rightarrow 0$ as $n\rightarrow \infty$. Then $x\in \mathcal{F}(\mathcal{S}).$
\end{lemma}

\section{Main results}

Now, we are ready to give our main results.

\begin{theorem}\label{x} Let $\mathcal{H}$ be a real Hilbert space and $\mathcal{S}:\mathcal{H}\rightarrow \mathcal{H}$, $\mathfrak{J}^{'}:\mathcal{H}\rightarrow {\mathcal{H}}$ be the nonexapansive mappings such that $x^*\in\Sigma=\mathcal{F}(\mathcal{S})\cap\Omega\neq\phi$. Let $\{x_n\}$ be a sequence generated in the following manner
	\begin{equation}\label{3.1}
	\begin{cases}
	w_n=x_n+\theta_n(x_n-x_{n-1})\\
	y_n=(1-\beta_n)w_n+\beta_n\mathcal{S}w_n\\
	x_{n+1}=(1-\alpha_n)\mathcal{S}y_n+\alpha_n\mathfrak{J}^{'}y_n
	\end{cases}
	\end{equation}
	where, $\mathfrak{J}^{'}=\mathcal{U}$ $(I+\gamma\mathcal{A}^*(\mathcal{V}-I)\mathcal{A}),$\\
	$\mathcal{U}=\mathfrak{J}^{\mathcal{B}_1}_\lambda(I-\lambda g_1)$ and $\mathcal{V}=\mathfrak{J}^{\mathcal{B}_2}_\lambda(I-\lambda g_2)$,
	$ ~~\lambda>0,~~\gamma\in \Big(0,\frac{2}{\|A^*A\|}\Big)$ and $n\geq1$. 
	Also, the sequences $\{\theta_n\}, \{\alpha_n\}$ and $\{\beta_n\}$  satisfy:
\begin{item}
\item (D1) $\sum_{n=1}^{\infty}\theta_n<\infty, \{\theta_n\}\subset[0,\theta]$, $ 0\leq\theta<1$;
\item (D2) $\{\alpha_n\}$, $\{\beta_n\} \subset [\delta,1-\delta]$ for some $\delta\in(0,0.5);$
\item (D3) $\{\mathcal{S}(w_n)-w_n\}$, $\{\mathcal{S}(w_n)-x^*\}$  are bounded;
\item (D4) $\{\mathfrak{J}^{'}(w_n)-w_n\}$, $\{\mathfrak{J}^{'}(w_n)-x^*\}$ are bounded for any $x^*\in \Sigma$.
\end{item}. 
Then  $\{x_n\}$ generated in the above iterative process holds,
\begin{enumerate}
\item $\lim_{n\rightarrow \infty}\|x_n-x^*\|$ exists.\label{1}
\item $ \lim_{n\rightarrow \infty}\|x_n-\mathcal{S}(x_n)\|=0=\lim_{n\rightarrow \infty}\|x_n-\mathfrak{J}^{'}(x_n)\|.$\label{2}
	\end{enumerate}

\end{theorem}

\begin{proof} Using triangle inequality, we have
		\begin{align}
		\|y_n-x^*\|&=\|(1-\beta_n)w_n+\beta_n\mathcal{S}(w_n)-x^*\|\nonumber\\
		&\leq (1-\beta_n)\|w_n-x^*\|+\beta_n\|\mathcal{S}(w_n)-x^*\|\nonumber\\
		&\leq (1-\beta_n)\|w_n-x^*\|+\beta_n\|w_n-x^*\|\nonumber\\
		&= \|w_n-x^*\|.\label{3.2}
		\end{align}
		So,
		\begin{align}
		\|x_{n+1}-x^*\|&=\|(1-\alpha_n)\mathcal{S}(y_n)+\alpha_n\mathfrak{J}^{'}(y_n)-x^*\|\nonumber\\
		&\leq(1-\alpha_n)\|\mathcal{S}(y_n)+x^*\|+\alpha_n\|\mathfrak{J}^{'}(y_n)-x^*\|\label{3.3}
		\end{align}
		Using the nonexpansiveness of $\mathcal{S},\mathfrak{J}^{'}$ and (\ref{3.2}), we have
		\begin{align}
		\|x_{n+1}-x^*\|&\leq (1-\alpha_n)\|\mathcal{S}(y_n)-x^*\|+\alpha_n\|\mathfrak{J}^{'}(y_n)-x^*\|\nonumber\\
		&\leq(1-\alpha_n)\|\mathcal{S}(y_n)-\mathcal{S}(x^*)\|+\alpha_n\|\mathfrak{J}^{'}(y_n)-\mathfrak{J}^{'}(x^*)\|\nonumber\\
		&\leq (1-\alpha_n)\|y_n-x^*\|+\alpha_n\|y_n-x^*\|\nonumber\\
		&\leq\|y_n-x^*\|.\label{3.4}
		\end{align}
		From (\ref{3.2}) and (\ref{3.4}), we have
		\begin{align}
		\|x_{n+1}-x^*\|&\leq \|y_n-x^*\|\leq\|w_n-x^*\|\nonumber\\
		\|x_{n+1}-x^*\|&\leq\|w_n-x^*\|.\label{3.5}
		\end{align}
		Now, using triangle inequality and conditions (D3) and (D4), we have
		\begin{align}
		\|w_n-x^*\|&=\|w_n-\mathcal{S}(w_n)+\mathcal{S}(w_n)-x^*\|\nonumber\\
		&\leq\|\mathcal{S}(w_n)-w_n\|+\|\mathcal{S}(w_n)-x^*\|\nonumber\\
		&\leq K_1 + K_2=K\nonumber
		\end{align}
		Finally, we conclude that
		\begin{align}
		\|w_n-x^*\|&\leq K
		\end{align}
		For some K$\in [0,\infty)$. That is $\{w_n\}$ is bounded. Hence by (\ref{3.5}), $\{x_n\}$ is bounded.\\
		By using the identity in equation (\ref{2.1}),
		\begin{align}
		\|w_n-x^*\|^2&=\|(1+\theta_n)(x_n-x^*)-\theta_n(x_{n-1}-x^*)\|^2\nonumber\\
		&=(1+\theta_n)\|x_n-x^*\|^2-\theta_n\|x_{n-1}-x^*\|^2+\theta_n(1+\theta_n)\|x_n-x_{n-1}\|^2\label{3.7}
		\end{align}
		From inequalities (\ref{3.5}) and (\ref{3.7})
		\begin{align}
		\|x_{n+1}-x^*\|^2&\leq\|w_n-x^*\|^2\nonumber\\
		&=(1+\theta_n)\|x_n-x^*\|^2-\theta_n\|x_{n-1}-x^*\|^2+\theta_n(1+\theta_n)\|x_n-x_{n-1}\|^2.\label{3.8}
		\end{align}
		Denote $\psi_n :=\|x_n-x^*\|^2$. Then (\ref{3.8}) becomes $\psi_{n+1}\leq\psi_n+\theta_n(\psi_n-\psi_{n-1})+\delta_n$, where $\delta_n=\theta_n(1+\theta_n)\|x_n-x_{n-1}\|^2$. Observe that by (D1),
		\begin{align}
		\sum_{n=1}^{\infty}\delta_n&=\sum_{n=1}^{\infty}\theta_n(1+\theta_n)\|x_n-x_{n-1}\|^2\nonumber\\
		&\leq\sum_{n=1}^{\infty}\theta_n(1+\theta_n)(2K)^2<\infty.
		\end{align}
		By lemma \ref{2.2} (2), there exists $\psi^* \in [0,\infty)$ such that $\lim_{n\rightarrow \infty}\psi_n=\psi^*$. This means that $\lim_{n\rightarrow \infty}\|x_n-x^*\|^2$ exists and, therefore,  $\lim_{n\rightarrow \infty}\|x_n-x^*\|$ exists. This completes the proof of (1).\\
		Now, to achieve (2) part of the theorem \ref{3.1}\\
		Set c = $\lim_{n\rightarrow \infty}\|x_n-x^*\|$. By the nonexpansiveness of $\mathcal{S}$ and $\mathfrak{J}^{'}$, we get
		\begin{align}
		\|x_n-\mathcal{S}(x_n)\|&\leq\|x_n-x^*\|+\|\mathcal{S}(x_n)-x^*\|\nonumber\\
		&\leq\|x_n-x^*\|+\|x_n-x^*\|\nonumber\\
		&=2\|x_n-x^*\|\label{3.10}
		\end{align}
		Also,
		\begin{align}
		\|x_n-\mathfrak{J}^{'}(x_n)\|&\leq\|x_n-x^*\|+\|\mathfrak{J}^{'}(x_n)-x^*\|\nonumber\\
		&\leq\|x_n-x^*\|+\|x_n-x^*\|\nonumber\\
		&=2\|x_n-x^*\|\label{3.11}
		\end{align}
		So, if $c=0$, then $\|x_n-\mathcal{S}(x_n)\|\rightarrow 0$, $\|x_n-\mathfrak{J}^{'}(x_n)\|\rightarrow 0$. Now assume that $c>0$. Note that $\sum_{n=1}^{\infty}\theta_n<\infty $ implies $\lim_{n\rightarrow \infty}\theta_n=0$. It follows from (\ref{3.7})
		\begin{align}
		\lim_{n\rightarrow \infty}\|w_n-x^*\|^2&=\lim_{n\rightarrow \infty}\|(1+\theta_n)(x_n-x^*)-\theta_n(x_{n-1}-x^*)\|^2\nonumber\\
		&=\lim_{n\rightarrow \infty}((1+\theta_n)\|x_n-x^*\|^2-\theta_n\|x_{n-1}-x^*\|^2+\theta_n(1+\theta_n)\|x_n-x_{n-1}\|^2)\nonumber\\
		&=\lim_{n\rightarrow \infty}\|x_n-x^*\|^2=c^2.
		\end{align}
		That is, $\lim_{n\rightarrow \infty}\|w_n-x^*\|=c$. So, this forces\\
		$\limsup_{n\rightarrow \infty}\|y_n-x^*\|\leq\limsup_{n\rightarrow \infty}\|w_n-x^*\|=c$. Next we will claim that $\liminf_{n\rightarrow \infty}\|y_n-x^*\|\geq c$. Since $\mathcal{S}$ and $\mathfrak{J}^{'}$ are nonexpansive, by (\ref{2.1}) we have
		\begin{align}
		\|x_{n+1}-x^*\|^2&=\|(1-\alpha_n)\mathcal{S}(y_n)+\alpha_n\mathfrak{J}^{'}(y_n)-x^*\|^2\nonumber\\
		&=(1-\alpha_n)\|\mathcal{S}(y_n)-x^*\|^2+\alpha_n\|\mathfrak{J}^{'}(y_n)-x^*\|^2-\alpha_n(1-\alpha_n)\|\mathcal{S}(y_n)-\mathfrak{J}^{'}(y_n)\|^2\nonumber\\
		&\leq(1-\alpha_n)\|\mathcal{S}(y_n)-x^*\|^2+\alpha_n\|\mathfrak{J}^{'}(y_n)-x^*\|^2\nonumber\\
		&\leq(1-\alpha_n)\|\mathcal{S}(y_n)-\mathcal{S}(x)^*\|^2+\alpha_n\|\mathfrak{J}^{'}(y_n)-\mathfrak{J}^{'}(x^*)\|^2\nonumber\\
		&\leq(1-\alpha_n)\|y_n-x^*\|^2+\alpha_n\|y_n-x^*\|^2\nonumber\\
		&\leq\|y_n-x^*\|^2\nonumber\\
		\|x_{n+1}-x^*\|^2&\leq\|y_n-x^*\|^2
		\end{align}
		Now, consider
		\begin{align}
		\|y_n-x^*\|^2&=\|(1-\beta_n)w_n+\beta_n\mathcal{S}(w_n)-x^*\|^2\nonumber\\
		&=(1-\beta_n)\|w_n-x^*\|^2+\beta_n\|\mathcal{S}(w_n)-x^*\|^2-\beta_n(1-\beta_n)\|w_n-\mathcal{S}(w_n)\|\nonumber\\
		&\leq(1-\beta_n)\|w_n-x^*\|^2+\beta_n\|\mathcal{S}(w_n)-x^*\|^2\nonumber\\
		&\leq(1-\beta_n)\|w_n-x^*\|^2+\beta_n\|\mathcal{S}(w_n)-\mathcal{S}(x^*)\|^2\nonumber\\
		&\leq\|w_n-x^*\|^2
		\end{align}
		Which yields $\liminf_{n\rightarrow \infty}\|y_n-x^*\|^2\geq c^2$ and so $\liminf_{n\rightarrow \infty}\|y_n-x^*\|\geq c$\\
		Since
		\begin{align}
		c\leq\liminf_{n\rightarrow \infty}\|y_n-x^*\|\leq\limsup_{n\rightarrow \infty}\|y_n-x^*\|\leq c,\nonumber
		\end{align}
		It follows that $\lim_{n\rightarrow \infty}\|y_n-x^*\|=c$.\\
		Again, using the nonexpansiveness of $\mathcal{S}$ and lemma \ref{n}, we have
		\begin{align}
		\|\mathcal{S}(w_n)-\mathcal{S}(x^*)\|&\leq\|w_n-x^*\|,\nonumber\\
		\limsup_{n\rightarrow\infty}\|\mathcal{S}(w_n)-\mathcal{S}(x^*)\|&\leq\limsup_{n\rightarrow\infty}\|w_n-x^*\|,\nonumber\\
		\limsup_{n\rightarrow\infty}\|\mathcal{S}(w_n)-x^*\|&\leq\limsup_{n\rightarrow\infty}\|w_n-x^*\|\leq c,\nonumber\\
		\lim_{n\rightarrow \infty}\|(1-\beta_n)(w_n-x^*)+\beta_n(\mathcal{S}(w_n)-x^*)\|&=\lim_{n\rightarrow \infty}\|y_n-x^*\|=c,\label{3.17}
		\end{align} 
		Also, using the nonexpansiveness of $\mathfrak{J}^{'}$ and lemma (\ref{n}), we have
		\begin{align}
		\|\mathfrak{J}^{'}(y_n)-\mathfrak{J}^{'}(x^*)\|&\leq\|y_n-x^*\|,\nonumber\\
		\limsup_{n\rightarrow\infty}\|\mathfrak{J}^{'}(y_n)-\mathfrak{J}^{'}(x^*)\|&\leq\limsup_{n\rightarrow\infty}\|y_n-x^*\|,\nonumber\\
		\limsup_{n\rightarrow\infty}\|\mathfrak{J}^{'}(y_n)-x^*\|&\leq\limsup_{n\rightarrow\infty}\|y_n-x^*\|\leq c,\nonumber\\
		\lim_{n\rightarrow \infty}\|(1-\alpha_n)(\mathcal{S}y_n-x^*)+\alpha_n(\mathfrak{J}^{'}(y_n)-x^*)\|&=\lim_{n\rightarrow \infty}\|x_{n+1}-x^*\|=c.\label{3.18}
		\end{align}
		Applying lemma \ref{n} in Eq $(\ref{3.17})$ and $(\ref{3.18})$, we get
		\begin{align}
		\lim_{n\rightarrow \infty}\|\mathcal{S}(w_n)-w_n\|=0
		\end{align}
		and
		\begin{align}
		\lim_{n\rightarrow \infty}\|\mathcal{S}(y_n)-\mathfrak{J}^{'}(y_n)\|=0\label{20}
		\end{align}
		However, we know that $y_n-w_n=\beta_n(\mathcal{S}(w_n)-w_n)$ and $w_n-x_n=\theta_n(x_n-x_{n-1})$\\
		which yeild
		\begin{align}
		0&\leq\lim_{n\rightarrow \infty}\|y_n-w_n\|\nonumber\\
		&=\lim_{n\rightarrow \infty}\beta_n\|\mathcal{S}(w_n)-w_n\|\nonumber\\
		&\leq\lim_{n\rightarrow \infty}\|\mathcal{S}(w_n)-w_n\|=0.\label{21}
		\end{align}
		and
		\begin{align}
		\lim_{n\rightarrow \infty}\|w_n-x_n\|&=\lim_{n\rightarrow \infty}\theta_n\|x_n-x_{n-1}\|=0.\label{22}
		\end{align}
		\rm Note that by (D3) and $\theta_n \rightarrow 0$ we have
		\begin{align}
		\lim_{n\rightarrow \infty}\|\mathcal{S}(w_n)-x_n\|&\leq\lim_{n\rightarrow \infty}\|\mathcal{S}(w_n)-w_n\|+\lim_{n\rightarrow \infty}\theta_n\|x_n-x_{n-1}\|=0.\label{23}
		\end{align}
		It follows that, by (\ref{20}), (\ref{21}), (\ref{22}), (\ref{23}) and the nonexpansiveness of $\mathcal{S}$ and $\mathfrak{J}^{'}$ we have
		\begin{align}
		0&\leq\lim_{n\rightarrow \infty}\|\mathcal{S}(x_n)-x_n\|\nonumber\\
		&\leq\lim_{n\rightarrow \infty}\|\mathcal{S}(x_n)-\mathcal{S}(w_n)\|+\lim_{n\rightarrow \infty}\|\mathcal{S}(w_n)-x_n\|\nonumber\\
		&\leq\lim_{n\rightarrow \infty}\|x_n-w_n\|+\lim_{n\rightarrow \infty}\|\mathcal{S}(w_n)-x_n\|=0.
		\end{align}
		and
		\begin{align}
		0&\leq\lim_{n\rightarrow \infty}\|\mathfrak{J}^{'}(x_n)-x_n\|\nonumber\\
		&\leq\lim_{n\rightarrow \infty}\|\mathfrak{J}^{'}(x_n)-\mathfrak{J}(y_n)\|+\lim_{n\rightarrow \infty}\|\mathfrak{J}^{'}(y_n)-\mathcal{S}(w_n)\|+\lim_{n\rightarrow \infty}\|\mathcal{S}(w_n)-x_n\|\nonumber\\
		&\leq\lim_{n\rightarrow \infty}\|x_n-y_n\|+\lim_{n\rightarrow \infty}\|\mathfrak{J}^{'}(y_n)-\mathcal{S}(w_n)\|+\lim_{n\rightarrow \infty}\|\mathcal{S}(w_n)-x_n\|\nonumber\\
		&\leq\lim_{n\rightarrow \infty}\|x_n-w_n\|+\lim_{n\rightarrow \infty}\|w_n-y_n\|=0.
		\end{align}
		Therefore, $\lim_{n\rightarrow \infty}\|\mathcal{S}(x_n)-x_n\|=0=\lim_{n\rightarrow \infty}\|\mathfrak{J}^{'}(x_n)-x_n\|$ as desired.
\end{proof}
\begin{theorem}
Let	$\mathcal{H}$ be a real Hilbert space having Opial's property. Suppose that $\mathcal{S}:\mathcal{H}\rightarrow\mathcal{H}$ and $\mathfrak{J}^{'}:\mathcal{H} \rightarrow {\mathcal{H}}$ are two nonexpansive mappings with $\Sigma=\mathcal{F}(\mathcal{S})\bigcap\Omega\neq\phi$. Then the sequence $\{x_n\}$ in algorithm \ref{3.1} converges weakly to a common fixed point of $\mathcal{S}$ and $\mathfrak{J}^{'}$.
\end{theorem}
\begin{proof}
Let $x^* \in \Sigma$. By Theorem (\ref{x})(\ref{1}), $\lim_{n\rightarrow \infty}\|x_n-x^*\|$ exists. Hence $\{x_n\}$ is bounded. Let $\{x_{n_k}\}$ and
$\{x_{n_j}\}$ be subsequences of the sequence of $\{x_n\}$ with the two weak limits $q_1$ and $q_2$, respectively. By theorem \ref{3.1} (2), $\lim_{n\rightarrow \infty}\|x_{n_k}-\mathcal{S}(x_{n_k})\|=0$, $\lim_{n\rightarrow \infty}\|x_{n_k}-\mathfrak{J}^{'}(x_{n_k})\|=0$, $\lim_{n\rightarrow \infty}\|x_{n_j}-\mathcal{S}(x_{n_j})\|=0$ and $\lim_{n\rightarrow \infty}\|x_{n_j}-\mathfrak{J}^{'}(x_{n_j})\|=0$. By Lemma (\ref{2.5}), $\mathcal{S}(q_i)=q_i,  \mathfrak{J}^{'}(q_i)=q_i$ for $i=1,2$. That is, $q_1,q_2\in\Sigma$. Applying theorem(\ref{3.1})(\ref{1}) again, we have $\lim_{n\rightarrow \infty}\|x_n-q_1\|$ and $\lim_{n\rightarrow \infty}\|x_n-q_2\|$ exist and both $\{x_{n_k}\}$ and $\{x_{n_j}\}$ are sequences converging to $q_1$ and $q_2$, respectively. By lemma (\ref{2.2}), $q_1= q_2$.\\ Therefore, $\{x_n\}$ converges weakly to a common fixed point in $\Sigma$.
\end{proof}

\begin{theorem}\label{th.3.3}
let $\mathcal{H}$ is a uniformly convex Hilbert space, Suppose that $\mathcal{S}:\mathcal{H}\rightarrow \mathcal{H}$ and $\mathfrak{J}^{'}:\mathcal{H}\rightarrow \mathcal{H}$ are two nonexapansive mappings with $\Sigma=\mathcal{F}(\mathcal{S})\bigcap\Omega\neq\phi$ and satisfy the {\bf condition (B)}. Then the sequence $\{x_{n}\}$ in (\ref{3.1}) converges strongly to a common fixed point of $\mathcal{S}$ and $\mathfrak{J}^{'}$.
\end{theorem}
	\begin{proof}
		Let $y\in\Sigma $. Now by $(\ref{3.8})$, we get
		\begin{align}
		\inf_{y\in\Sigma}\{\|x_{n+1}-y\|^2\}&\leq \inf_{y\in\Sigma}\{\|w_{n}-y\|^2\}\nonumber\\
		&=\inf_{y\in\Sigma}\{(1+\theta_n)\|x_{n}-y\|^2\}+\inf_{y\in\Sigma}\{-\theta_n\|x_{n-1}-y\|^2\}\nonumber\\
		&+\inf_{y\in\Sigma}\{\theta_n(1+\theta_n)\|x_{n}-x_{n-1}\|^2\}\nonumber\\
		&\leq\inf_{y\in\Sigma}\{\|x_{n}-y\|^2\}+\theta_n\inf_{y\in\Sigma}\{\|x_{n}-y\|^2\}+\inf_{y\in\Sigma}\{-\theta_n\|x_{n-1}-y\|^2\}\nonumber\\
		&+\inf_{y\in\Sigma}\{\theta_n(1+\theta_n)\|x_{n}-x_{n-1}\|^2\}\nonumber\\
		&\leq\inf_{y\in\Sigma}\{\|x_{n}-y\|^2\}+\theta_n[\inf_{y\in\Sigma}\{\|x_{n}-y\|^2\}+\inf_{y\in\Sigma}\{-\theta_n\|x_{n-1}-y\|^2\}\nonumber\\
		&+\inf_{y\in\Sigma}]\{\theta_n(1+\theta_n)\|x_{n}-x_{n-1}\|^2\}\label{3.25}
		\end{align}
		Denote $\Psi_n:=\inf_{y\in\Sigma}\{\|x_n-y\|^2\}$.Then (\ref{3.25}) becomes
		\begin{align}
		\Psi_{n+1}\leq\Psi_n+\theta_n(\Psi_n-\Psi_{n-1})+\delta_n,
		\end{align}
		where $\delta_n=\theta_n(1+\theta_n)\|x_n-x_{n-1}\|^2$. Observe that by (D1),
		\begin{align}
		\sum_{n=1}^{\infty}\delta_n&=\sum_{n=1}^{\infty}\theta_n(1+\theta_n)\|x_n-x_{n-1}\|^2\nonumber\\
		&\leq\sum_{n=1}^{\infty}\theta_n(1+\theta_n)(2K)^2<\infty.\label{3.27}
		\end{align}
		By Lemma \ref{l,2.3}(2), there exist $\Psi^*\in[0,\infty)$ such that $\lim_{n\rightarrow \infty}\Psi_n=\Psi^*$.\\
		That is, $\lim_{n\rightarrow \infty}\inf_{y\in\Sigma}\{\|x_n-y\|^2\}$ exists. Therefore, $\lim_{n\rightarrow \infty}\inf_{y\in\Sigma}\{\|x_n-y\|\}$ exists. Since $\mathcal{S}$ and $\mathfrak{J}^{'}$ satisfy {\bf condition B}, by Theorem \ref{x}(\ref{1}) it implies that
		\begin{align*}
		\lim_{n\rightarrow \infty}f\big(inf_{y\in\Sigma}\{\|x_n-y\|\}\big)=0
		\end{align*}
		and, thus,
		\begin{align*}
		\lim_{n\rightarrow \infty}\inf_{y\in\Sigma}\{\|x_n-y\|\}=0
		\end{align*}
		\rm So, we can find a subsequence $\{x_{n_j}\}$ of $\{x_n\}$ and a sequence $\{x^*_j\}\subset\Sigma$ satisfying $\|x_{n_j}-x^*_j\|<\frac{1}{2^j}$. Next we will show that $\{x^*_j\}$ is a cauchy sequence. Let $\epsilon>0$. Since $\lim_{n\rightarrow \infty}\inf_{y\in\Sigma}\{\|x_n-y\|\}=0$, there is $N\in\mathbb{N}$ such that $\inf_{y\in\Sigma}<\frac{\epsilon}{6}$ for all $n\geq N$. For all $m,n\geq N$, we have
		\begin{align*}
		\|x_m-x_n\|\leq\|x_m-y\|+\|x_n-y\|
		\end{align*}
		for all $y\in\Sigma$.Thus,
		\begin{align*}
		\|x_m-x_n\|&\leq\inf_{y\in\Sigma}\{\|x_m-y\|+\|x_n-y\|\}=\inf_{y\in\Sigma}\{\|x_m-y\|\}+\inf_{y\in\Sigma}\{\|x_n-y\|\} <\frac{\epsilon}{6}+\frac{\epsilon}{6}=\frac{\epsilon}{3}
		\end{align*}
		for all $\rm m,n\geq N$. Also, there is $j_0\in \mathbb{N}$ such that $\frac{1}{2^{j_0}}<\frac{\epsilon}{3}$. Choose $\rm M=max\{N,j_0\}$. Then, for all $j>k\geq M$, we have\\
		\begin{align*}
		\|x^*_j-x^*_k\|&\leq\|x^*_j-x_{n_j}\|+\|x_{n_j}-x_{n_k}\|+\|x_{n_k}-x^*_k\|<\frac{\epsilon}{6}+\frac{\epsilon}{6}=\frac{\epsilon}{3}=\epsilon.
		\end{align*}
		Therefore, $\{x^*_j\}$ is a cauchy sequence and so there exist $q\in\mathcal{H}$ such that $x^*_j$ converges to q. Since $\Sigma$ is closed, $q\in \Sigma$. As a result , we see that $x_{n_j}$ converges to q. Since $\lim_{n\rightarrow \infty}\|x_n-q\|$ exist by Theorem \ref{x}(1), the conclusion follows.
	\end{proof}
\begin{theorem}
	Let $\mathcal{H}$ is a uniformly convex Hilbert space, Suppose that $\mathcal{S},\mathcal{J}^{'}:\mathcal{H}\rightarrow\mathcal{H}$ are two nonexpansive mapping with $\Sigma=\mathcal{F}(\mathcal{S})\cap\Omega \neq\phi$ and $\mathcal{S}$ is semicompact. Then the sequence $\{x_n\}$ in (\ref{x})converges strongly to a common fixed point of $\mathcal{S}$ and $\mathcal{J}^{'}$.
	\begin{proof}
From theorem \ref{3.1}, $\{x_n\} $ is bounded and $\liminf_{n\rightarrow \infty}\|x_n-\mathcal{S}(x_n)\|=0=\lim_{n\rightarrow \infty}\|x_n-\mathcal{J}^{'}\|$. By the semicompactness of $\mathcal{S}$, there exists $q\in\mathcal{H}$ and a subsequence $\{x_{n_j}\}$
of $\{x_n\}$ such that $x_{n_j}\rightarrow q$ as $j\rightarrow \infty$. Then
\begin{align}
\|q-\mathcal{S}(q)\|&\leq\|q-x_{n_j}\|+\|x_{n_j}-\mathcal{S}(x_{n_j})\|+\|\mathcal{S}(x_{n_j})-\mathcal{S}(q)\|\nonumber\\
&\leq\|q-x_{n_j}\|+\|x_{n_j}-\mathcal{S}(x_{n_j})\|+\|x_{n_j}-\mathcal{S}(q)\|\nonumber\\
&\rightarrow 0 ~~~ as~~~j\rightarrow \infty. 
\end{align}
Also,
\begin{align}
\|q-\mathcal{J}^{'}(q)\|&\leq\|q-x_{n_j}\|+\|x_{n_j}-\mathcal{J}^{'}(x_{n_j})\|+\|\mathcal{J}^{'}(x_{n_j})-\mathcal{J}^{'}(q)\|\nonumber\\
&\leq\|q-x_{n_j}\|+\|x_{n_j}-\mathcal{J}^{'}(x_{n_j})\|+\|x_{n_j}-\mathcal{J}^{'}(q)\|\nonumber\\
&\rightarrow 0 ~~~ as~~~j\rightarrow \infty. 
\end{align}
Thus $q\in\Sigma$. As in the proof of the theorem \ref{th.3.3}, $\lim_{n\rightarrow \infty}\inf_{y\in\Sigma}\{\|x_n-y\| \}$ exists. We observe that $\inf_{y\in\Sigma}\{\|x_{n_j}-y \|\}\leq\|x_{n_j}-q\|\rightarrow 0$ as $j\rightarrow \infty,$ hence $\liminf_{n\rightarrow \infty}\inf_{y\in\Sigma}\{\|x_n-y\| \}=0$. It follows, as in the proof theorem \ref{th.3.3}, that $\{x_n\}$ converges strongly a common fixed point $\mathcal{S}$ and $\mathcal{J}^{'}$. This completes the proof.
\end{proof}
\end{theorem}
\section{Numerical illustration}
\textbf{Example 1.} Let $\mathcal{H}_1=\mathcal{H}_2=\mathbb{R}$, the set of all real numbers, with inner product defined by $\langle x,y\rangle=xy$, $\forall x,y \in \mathbb{R}$. Let $\mathcal{U}(x)= \frac{9x-7}{2}$, $\mathcal{V}(x)=\frac{5x-2}{3}$ and $A:\mathbb{R}\rightarrow\mathbb{R}$ be defined by $A(x)=-\frac{x}{2}, \forall x\in\mathbb{R}$. let the mapping $S:C\rightarrow\mathbb{R}$ be defined by $S(x)=\frac{x+3}{4},\forall x\in C$. Setting $\theta_n=0.98$, $\gamma = 0.000025$, $\alpha_n=\frac{1}{n+1}$ and $\beta_n= \frac{n+1}{n+2}$. The sequence $\{x_n\}$ generated by algorithm (\ref{3.1}) reduces to\\
\begin{equation}\label{4.1}
\begin{cases}
w_n=x_n+0.98(x_n-x_{n-1})\\
y_n=\frac{1}{n+2}w_n+\frac{n+1}{n+2}\mathcal{S}w_n\\
x_{n+1}=\frac{n}{n+1}\mathcal{S}y_n+\frac{1}{n+1}\mathfrak{J}^{'}y_n
\end{cases}
\end{equation}
\\
Then, the sequence $\{x_n\}$ converges to a point $x^*=1\in\Sigma$
\newpage
\begin{table}
	\centering
	\begin{tabular}{||p{5cm}|p{5cm}|p{5cm}|p{5cm}|} 
		\hline\hline
		\cellcolor{lime} Number of iterations & $ x_0= -0.1 $ \cellcolor{green} &\cellcolor{orange} $x_0=0.1$ \\ 
		[4ex]
		\hline
		\rowcolor{lightgray}1. &-0.100000                 &  0.100000\\           
		\rowcolor{lightgray}2. &-0.100000                &  0.840001\\
		\rowcolor{lightgray}3. &0.804445                 & 1.092317\\
		\rowcolor{lightgray}4. &1.112832                 & 1.052322\\
		\rowcolor{lightgray}5. &1.063949                 & 1.001938\\
		\rowcolor{lightgray}6. &1.002369                 & 0.993224\\
		\rowcolor{lightgray}7. &0.991718                 & 0.997868\\
		\rowcolor{lightgray}8. &0.997395                 & 1.000330\\
		\rowcolor{lightgray}9. &1.000403                 & 1.000368\\
		\rowcolor{lightgray}10. &1.000449                 & 1.000054\\
		\rowcolor{lightgray}11. &1.000066                 & 0.999967\\
		\rowcolor{lightgray}12. &0.999960                 & 0.999985\\
		\rowcolor{lightgray}13. &0.999982                 & 1.000001\\
		\rowcolor{yellow}14. &1.000000                 & 1.000000\\
		\hline
	\end{tabular}

\label{table:nonline}

\caption
{Numeric result for the above example, shows the convergence of iterative scheme $x_n$ to $x^*=1\in\Sigma$}

\end{table}

\begin{figure}
	\begin{center}
		\includegraphics[width=17cm, height=9cm]{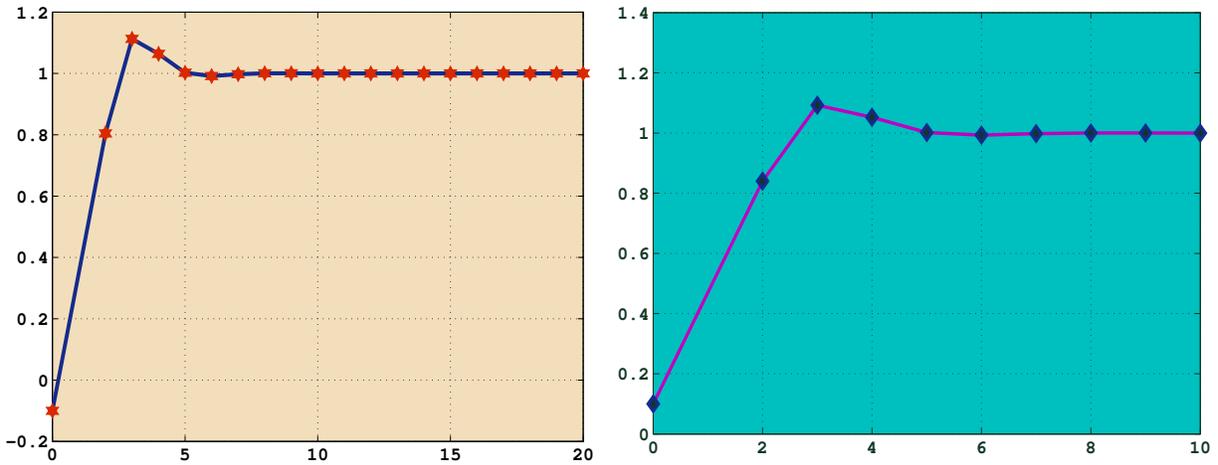}
		\caption{Shows the convergence of sequence $\{x_n\}$ to $x^*=1$ for intial values $x_0=-0.1$ in fig(1) and $x_0=0.1$ in fig(2).}
		\bf \label{fig 1and2:graphe}
	\end{center}
\end{figure}

\vskip 6mm
\newpage
\noindent{\bf Conclusion:}
\noindent   In this paper, we find a common solution of split monotone inclusion problem and fixed point problem by modified S-iteration process combining with inertial step. Weak and strong convergence theorems are established under standard assumptions imposed on cost operators and mappings. Finally, several preliminary,  numerical experiments have also been performed to illustrate the convergence of the proposed algorithms.
\vskip 6mm
\noindent{\bf Disclosure Statement}

\noindent No potential conflict of interest was reported by the authors.

\end{document}